\documentclass[12pt,reqno]{amsart}
\usepackage[foot]{amsaddr}
\usepackage{geometry}
\usepackage{amsmath,amssymb,amsthm,amsfonts,enumerate,thmtools,thm-restate}
\usepackage{graphicx,pgfplots}
\usepackage[mathscr]{euscript}
\usepackage{mathtools}
\usepackage{tikz-cd}
\usepackage{old-arrows}
\pgfplotsset{compat=1.18}
\usepackage[colorlinks=true,linkcolor=blue,citecolor=blue]{hyperref}

\theoremstyle{plain}
\newtheorem{lemma}{Lemma}[section]
\newtheorem{theorem}[lemma]{Theorem}

\newtheorem{proposition}[lemma]{Proposition}

\theoremstyle{definition}
\newtheorem{definition}[lemma]{Definition}
\newtheorem{remark}[lemma]{Remark}
\newtheorem{example}[lemma]{Example}
\newtheorem{theoremmain}{Theorem}
\newtheorem{cormain}[theoremmain]{Corollary}

\graphicspath{ {images/} }

\newcommand{\norm}[1]{\left\lVert #1 \right\rVert}

\def\eps{\varepsilon}
\def\Ri{\mathcal R}

\def\C{\mathscr C}
\def\R{\mathbb{R}}
\def\S{\mathbb{S}}
\def\K{\mathscr{K}}
\def\L{\mathscr{L}}

\def\rad{\operatorname{rad}}
\def\diam{\operatorname{diam}}
\def\cat{\mathrm{CAT}}
\def\sd{\operatorname{sd}}
\usetikzlibrary{arrows,decorations.pathmorphing,backgrounds,calc,math}
\tikzstyle{map}=[->,semithick]
\tikzstyle{arc}=[bend left,->,semithick]
\tikzstyle{rinclusion}=[right hook->,semithick]
\tikzstyle{linclusion}=[left hook->,semithick]

\title[Vietoris--Rips Complexes Near subspaces with curvature $\leq \kappa$]{Topological Stability and Latschev-type Reconstruction Theorems for Spaces\\ of Curvature Bounded Above}

\author{Rafal Komendarczyk}
\address{Mathematics Department, Tulane University, USA}
\email{rako@tulane.edu}

\author{Sushovan Majhi}
\address{Data Science Program, George Washington University, USA}
\email{s.majhi@gwu.edu}

\author{Will Tran}
\address{Mathematics Department, Southwestern University, USA}
\email{tranw@southwestern.edu}


\begin{document}
\begin{abstract}
We consider the problem of homotopy-type reconstruction of compact subsets $X\subset\R^N$ that have the Alexandrov curvature bounded above ($\leq$ $\kappa$) in the intrinsic length metric. The reconstructed spaces are in the form of Vietoris--Rips complexes computed from a compact sample $S$, Hausdorff--close to the unknown shape $X$. Instead of the Euclidean metric on the sample, our reconstruction technique leverages a path-based metric to compute these complexes. As naturally emerging in the framework of reconstruction,  we also study the Gromov--Hausdorff topological stability and finiteness problem for general compact for subspaces of curvature bounded above. 
Our techniques provide novel sampling conditions as an alternative to the existing and commonly used techniques using weak feature size and $\mu$--reach. To the best of our knowledge, this is the first work that establishes homotopy-type reconstruction guarantees for spaces with vanishing reach and 
$\mu$--reach, a regime not covered by existing sampling conditions.
\end{abstract}

\maketitle

\setcounter{tocdepth}{1}
\tableofcontents

\vspace{-1cm}

\section{Introduction}
In the last two decades, the problem of shape reconstruction has received increasing attention in theoretical and applied communities alike.
Although the shape reconstruction problem comes in many forms, our study focuses on the reconstruction of a homotopy-type for both \emph{abstract} and \emph{Euclidean embedded} shapes under the Gromov--Hausdorff $d_{GH}$ and Hausdorff distance $d_H$, respectively.
The former type models the \emph{hidden} underlying shape as an abstract metric space $(X,d_X)$ and a sample $(S,d_S)$ as a (possibly finite) metric space having a small Gromov--Hausdorff (Definition~\ref{def:gh}) to $X$.  
Whereas, the latter type assumes the unknown shape to be a Euclidean subset $X\subset\mathbb{R}^N$ and the sample a Euclidean subset $S \subset \mathbb{R}^N$ close to $X$ in the Hausdorff distance (Definition~\ref{def:dH}).
In both cases, the goal is to produce (from the sample $S$) a topological space $\widetilde{X}$---called the \emph{reconstruction} of the underlying shape $X$---such that $\widetilde{X}$ is equivalent to $X$ in a given topological sense, i.e. typically up to a {\em homeomorphism} $\cong$, {\em diffeomorphism}\footnote{in the case $X$ is a differentiable manifold}, or a {\em homotopy equivalence} $\simeq$, where the third is the weakest among all of them\footnote{diffeomorphism $\Rightarrow$ homeomorphism $\Rightarrow$ homotopy equivalence}. 

If the underlying space $X$ is a differentiable manifold\footnote{realized via a specific embedding into $\R^N$}, there is a large body of literature on the reconstruction problem and well-developed algorithms for the extraction of $\widetilde{X}$. They are commonly referred to as {\em manifold learning} techniques as they find direct applications to machine learning, more widely known as {\em AI}. In this brief introduction, we only point to a fraction of existing literature in \cite{tenenbaum2000isomap,fefferman2020reconstructionI, fefferman2020intrinsic,fefferman2023fitting,genovese2012estimation,boissonnat2014delaunay,fefferman2016testing,zha2007continuum}, where the authors consider the problem of $C^2$--manifold reconstruction based on a finite, possibly noisy sample. The key property of $n$--manifolds\footnote{\cite{tenenbaum2000isomap,fefferman2020reconstructionI, fefferman2020intrinsic,genovese2012estimation,boissonnat2014delaunay} focus primarily on closed $n$--manifolds} which allows for their effective reconstruction \cite{fefferman2020reconstructionI, fefferman2020intrinsic,genovese2012estimation} is the fact that they are defined via a {\em differentiable structure}: a set of smooth charts (or local parametrizations whose domains are open subsets of $\R^n$) known as {\em atlas}. Equivalently, Euclidean submanifolds can be obtained as regular level sets of $\R^m$--valued functions. Such a convenient description is not available for more general shapes that manifest irregular features like sharp corners, branches, cusps, and multi-dimensional stratification. 
Examples include embedded graphs\footnote{for example {\em road networks}} (or more generally simplicial complexes), pinched manifolds, manifolds with boundary, stratified manifolds with strata of various dimensions, or more general compact Euclidean subsets. We also note that closed submanifolds of $\R^N$ have a positive reach $r(X)$ \cite{federer}, which is a critical property required by the existing manifold reconstruction schemes \cite{fefferman2020reconstructionI, fefferman2020intrinsic,genovese2012estimation}. However, the irregular shapes mentioned above, typically do not meet this assumption. The class of irregular shapes considered in the current work is known, in metric geometry \cite{gromov1999metric, burago2022course, bridson1999metricspaces}, as  spaces of Alexanvrov curvature bounded above (see Section 
\ref{S:cat(k)}). For such irregular shapes, lacking a differentiable structure, one can only hope to obtain the reconstruction $\widetilde{X}$ as a geometric realization of an abstract simplicial complex, homeomorphic or homotopy equivalent to $X$.

Each reconstruction technique relies on a set of \emph{sampling parameters}, which usually depends on the geometry and topology of the hidden shape $X$. 
Some of the above mentioned reconstruction works use reach $r(X)$ \cite{fefferman2020intrinsic,fefferman2023fitting,genovese2012estimation,niyogi2008finding,chazal2008smooth,majhi2024demystifying} and for more general shapes, the $\mu$--reach $r_\mu(X)$ \cite{chazal2006sampling,attali2011vietoris,kim2020homotopy}; see Section \ref{S:mu-reach} for a comparison.
Alternatively, other works \cite{fasy2022reconstruction,majhi2023vietoris}  use convexity radius $\rho(X)$ (Definitions~\ref{def:conv-rad}) and global distortion $\delta(X)$ of the embedding (Definition~\ref{def:dist}).
In this paper, we generalize the already existing notion of global distortion $\delta(X)$ (Definition~\ref{def:dist}) \cite{gromov1999metric} to introduce a novel sampling parameter: \emph{large scale distortion} $\delta^\eps_R(X)$ for $\eps, R >0$. The primary advantage of the new parameter $\delta^\eps_R(X)$ is that
it attains a finite value for various shapes with vanishing $\mu$--reach (Section \ref{S:mu-reach}). 
Our main Euclidean shape reconstruction theorem is the following\footnote{For convenience, the reader may assume $\xi=\tfrac{1}{14}$ in all the theorems presented here. }.
\begin{restatable}[Reconstruction of Subspaces with Curvature $\leq \kappa$ in $\R^N$]{theoremmain}{mainRips}
\label{thm:main-rips}
Let $X$ be a compact subspace of $\R^N$ with with curvature $\leq \kappa$ and let $\xi\in\left(0,\tfrac{1}{14}\right]$ be a fixed number. For any $\beta$ satisfying $0<\beta<\frac{1}{1+2\xi}\Delta(X)$,  
if we choose
$0<\eps\leq\beta$ such that $\delta^\eps_{2\xi\beta}(X)\leq1+\left(\frac{\xi}{1+\xi}\right)$, then for any compact subset $S\subset\R^N$
\[
d_H(X,S)<\tfrac{1}{2}\xi\eps\quad \Rightarrow\quad \Ri^\eps_\beta(S)\simeq X.
\]
\end{restatable}
The reconstructed shape, denoted as $\widetilde{X}=\Ri^\eps_\beta(S)$, presented in the theorem above, corresponds to the geometric representation of the Vietoris–Rips complex associated with a metric space $(S,d^\eps_S)$. Here, $d^\eps_S$ is defined as a path metric (see Definition~\ref{def:d-eps}) for the sample metric space  $(S,d^\eps_S)$. The parameter $\Delta(X)$ incorporates both the convexity radius $\rho(X)$ and the curvature limit $\kappa$ of the space $X$, as specified in~\eqref{eq:Delta(X)}. The adoption of the path metric is essential because traditional Euclidean Vietoris–Rips (and \v{C}ech) complexes often do not preserve topological properties, a limitation highlighted in works such as \cite{fasy2022reconstruction}. In particular, Theorem~\ref{thm:main-rips} represents, to the best of our knowledge, the first reconstruction theorem applicable to compact $\mathrm{CAT}(\kappa)$ subspaces with vanishing reach or $\mu$--reach, a geometric regime in which all classical manifold- and reach-based methods are not applicable. A key ingredient enabling this advance is the introduction of the large-scale distortion parameter $\delta^\eps_{2\xi\beta}(X)$, a new quantitative tool that captures intrinsic geometric stability at mesoscale resolutions. This parameter plays a critical role in extending Latschev-type arguments to non-smooth Alexandrov spaces and is essential for overcoming the absence of smooth structure and injectivity-radius bounds. Together, these ideas form the backbone of Theorem~\ref{thm:main-rips} and substantially broaden the scope of topological reconstruction theory.

Based on the inherent geometry of the shape $X$, Theorem \ref{thm:main-rips} specifies a range of values for the parameter $\beta$ that ensures the homotopy equivalence of these complexes to $X$. The demonstration of Theorem \ref{thm:main-rips} is twofold. Initially, it expands upon the findings from \cite{majhi2024demystifying} to include spaces with curvature $\leq \kappa$, thereby broadening the applicability of the well-established Latschev's Theorem, referenced in \cite{latschev2001vietoris}. 
\begin{restatable}[Quantitative Latschev's Theorem for Compact Spaces with curvature $\leq \kappa$]{theoremmain}{mainLatschev}
\label{thm:latschev-cat}
Let $(X,d_X)$ be a compact geodesic space with curvature $\leq \kappa$ and $\xi\in\left(0, \tfrac{1}{14}\right]$ a fixed number. 
For any $\beta$ satisfying
$0<\beta<\frac{\Delta(X)}{(1+2\xi)}$, if a compact metric space $(S,d_S)$ is a $(\xi\beta,\beta)$-Gromov--Hausdorff close to $X$, then $\Ri_\beta(S)\simeq X$. In particular, 
\[
d_{GH}(X,S)<\xi\beta\quad \Rightarrow\quad \Ri_\beta(S)\simeq X,
\]    
where $\Ri_\beta(S)$ is the geometric realization of the Vietoris--Rips complex of $(S,d_S)$ at scale $\beta$.
\end{restatable}
\noindent
Here, saying that $(S,d_S)$ is a $(\xi\beta,\beta)$–Gromov--Hausdorff close to $X$ means that there exists a correspondence $\C\subset X\times S$ such that for any $(x_1,s_1),(x_2,s_2)\in \C$ with \[
\min\{d_X(x_1,x_2),d_S(s_1,s_2)\}\le \beta,\]
 we have $|d_X(x_1,x_2)-d_S(s_1,s_2)|\le 2\xi\beta$ (Definition~\ref{def:eps-R-approx}).
 The proof of Theorem~\ref{thm:latschev-cat} is presented in Section~\ref{S:proof-ThmB} and is structurally similar to its predecessor (Theorem~\ref{thm:quan-latschev}) proven in \cite{majhi2024demystifying} since it relies on the well known Whitehead Theorem \cite{hatcher2002book}. A discernible difference is that the weaker\footnote{than the Gromov--Hausdorff closeness} notion of the $(\eps,R)$-Gromov--Hausdorff closeness  is employed in Theorem~\ref{thm:latschev-cat}. While this theorem provides foundational insights, it alone does not suffice for the comprehensive reconstruction results aimed for in Theorem \ref{thm:main-rips}, which stands as the centerpiece of this paper. 
 
 Nevertheless, Theorem~\ref{thm:latschev-cat} facilitates the exploration of topological stability within the class of spaces with curvature $\leq \kappa$.

It should be noted that Gromov--Hausdorff limits are generally not topologically faithful, as outlined in sources like \cite{gromov1999metric, burago2022course}. However, stability results have been established within specific classes of spaces. One prominent example is Gromov's Theorem \cite[Theorem 8.19]{gromov1999metric}, which guarantees stability (up to a diffeomorphism) for complete Riemannian $n$--manifolds that have two-sided sectional curvature bounds and a uniform minimum on their injectivity radii. 

Additionally, Alexandrov spaces characterized by a lower bound on curvature, fixed dimensions, and a diameter constraint also exhibit stability, in this case up to a homeomorphism, as demonstrated in studies like \cite{pereleman1992aleksandrov, perelman1993elementsmorse} and for manifolds in \cite{grove1990geometricfiniteness}. 
Thanks to Theorem \ref{thm:latschev-cat}, here we obtain the following homotopy-type stability result for spaces of curvature $\leq \kappa$. 

\begin{restatable}[Topological Stability of compact Spaces with Curvature $\leq \kappa$]{theoremmain}{mainStability}\label{thm:main-stability}
Let $X, X'$ be compact spaces with curvature $\leq \kappa$. 
For any $\xi\in\left(0,\tfrac{1}{14}\right]$, let $\beta$ be a number satisfying $0<\beta<\frac{1}{(1+2\xi)}\min\{\Delta(X)$, $\ \Delta(X')\}$.
If $X$ and $X'$ are $(\xi\beta,\beta)$- Gromov--Hausdorff close, then $X$ and $X'$ are homotopy equivalent. In particular, 
\[
d_{GH}(X,X')<\xi\beta\quad \Rightarrow\quad X\simeq X'.
\]
\end{restatable}
Homotopy-type stability in the setting of Riemannian manifolds and Finsler manifolds has previously been obtained in \cite{yamaguchi1988homotopyfiniteness} and
\cite{zhao2013homotopyfiniteness}. For general $\operatorname{LGC}$--spaces (i.e. locally geometrically contractible spaces) with bounded topological dimension, homotopy-type stability 
was obtained in \cite{petersen1990finiteness}. The spaces of curvature $\leq \kappa$, considered here, 
are $\operatorname{LGC}$--spaces (due to the lower bound for the convexity radius), however in contrast to results in \cite{petersen1990finiteness}, Theorem \ref{thm:main-stability} does not impose restrictions on the topological dimension of spaces.

\noindent A typical corollary of a stability result of this kind is a quantitative analog of the well-known Cheeger's finiteness theorem for Riemannian manifolds \cite{cheeger1970finiteness} (see also \cite{peters1984cheegers, grove1990geometricfinitness, petersen1990finiteness}). In particular, Theorem \ref{thm:main-stability} implies  

\begin{cormain}
In a class of compact spaces with curvature $\leq \kappa$, and a uniform lower bound for convexity radius, any precompact family admits only finitely many homotopy types.
\end{cormain}

The second component in the proof of Theorem \ref{thm:main-rips} is the following approximation result for compact geodesic subspaces of $\R^N$ by a Hausdorff-close Euclidean sample.

\begin{restatable}[$(\eps,R)$--Gromov--Hausdorff closeness Theorem]{theoremmain}{mainApprox}\label{thm:main-approx}
Let $X$ be a geodesic subspace of $\R^N$.
Let $\xi\in\left(0,1\right)$ and $\beta>0$ be fixed numbers. 
If we choose\footnote{it is possible to choose such a small $\eps$, thanks to Proposition \ref{prop:distortion_to_1}.}
$0<\eps\leq\beta$ such that $\delta^\eps_{2\xi\beta}(X)\leq1+\left(\frac{\xi}{1+\xi}\right)$, then for any compact subset $S\subset\R^N$ with $d_H(X,S)<\frac{1}{2}\xi\eps$, the metric space $(S,d^\eps_S)$ is an $(\xi\beta,\beta)$-Gromov--Hausdorff close to $(X,d^L_X)$.
\end{restatable}
A key element of Theorem \ref{thm:main-approx} is the employment of 
the path metric $d^\eps_S$ of the sample $S$. The notion of the path metric in this context was previously introduced in \cite{fasy2022reconstruction, majhi2023vietoris}. In this study, we further our understanding of these metrics. In particular, we show in Theorem~\ref{thm:path-metric-stab} that they are stable under Hausdorff perturbation. Moreover in Theorem~\ref{thm:geodesic-convergence}, the path metrics on a compact geodesic space $X$ have been shown to converge to the intrinsic metric of $X$. 

\subsection{Background on Latschev's Theorem}\label{S:background}
Hausmann proved in \cite{hausmann1995vietoris} that a closed Riemannian manifold is homotopy equivalent to its own Vietoris--Rips complex for a sufficiently small scale (see Theorem \ref{thm:hausmann}). The result quite naturally elicits the question of \emph{finite reconstruction}: Under what conditions is a Riemannian manifold homotopy equivalent to the Vietoris--Rips complex of a finite but dense sample? 
In \cite{latschev2001vietoris}, Latschev later provided an answer to the above manifold reconstruction problem satisfactorily, and also in a little more general context of density measured in the Gromov--Hausdorff distance (Definition~\ref{def:gh}). 
Latschev's theorem guarantees: 
For a closed Riemannian manifold $X$, there exists a positive number $\epsilon_0$ such that for any scale $0<\beta\leq\epsilon_0$ there exists some $\delta>0$ such that for every metric space $S$ with Gromov--Hausdorff distance to $X$ less than $\delta$, the Vietoris--Rips complex $\Ri_\beta(S)$ is homotopy equivalent to $X$. 
Although the result is purely qualitative and existential, its relevance can be ascribed to the emphasis that the quantity $\epsilon_0$ must depend solely on the intrinsic geometry of $X$.
Recently, the new quantitative version of Latschev's theorem has been proved by Majhi \cite{majhi2024demystifying}. 
\begin{theorem}[Quantitative Latschev's Theorem for Riemannian manifolds \cite{majhi2024demystifying}]\label{thm:quan-latschev}
Let $X$ be a closed, connected Riemannian manifold. 
Let $S$ be a compact metric space and $\beta>0$ a
number such that
\[
\frac{1}{\xi}d_{GH}(X,S)<\beta<\frac{\Delta(X)}{1+2\xi}
\] for some $\xi\in\left(0,\tfrac{1}{14}\right]$. Then, $\Ri_\beta(S)\simeq X$.
\end{theorem}
\noindent The result, indeed, \emph{demystifies} the constants in Latschev's theorem: 
\[
\epsilon_0=\frac{\Delta(X)}{1+2\xi}\text{ and }\delta=\xi\beta.
\]
In the same spirit, a \emph{Latschev-type} result was also presented in \cite{majhi2024demystifying} for the reconstruction of Euclidean submanifolds from a Hausdorff-close sample. 
The reach $r(X)$ of the manifold $X$ has been used as the sampling parameter; see \cite{federer} for definition.
\begin{theorem}[Submanifold Reconstruction under Hausdorff Distance \cite{majhi2024demystifying}]
\label{thm:quant-euclid} Let $X\subset\R^d$ be a closed, connected Euclidean submanifold. 
Let $S\subset\R^d$ be a compact subset and $\beta>0$ a number such that 
\[
\frac{1}{\xi}d_{H}(X,S)<\beta\leq\frac{3(1+2\xi)(1-14\xi)}{8(1-2\xi)^2}r(X)
\] 
for some $\xi\in\left(0,\tfrac{1}{14}\right)$. Then, $\Ri_\beta(S)\simeq X$.
Here, $\Ri_\beta(S)$ denotes the Vietoris--Rips complex of $S$ under the Euclidean distance.
\end{theorem}

\subsection{Related works on the reconstruction of homotopy type}
The homotopy-type reconstruction problem has been investigated for various shape classes, we list a fraction of the existing literature in \cite{niyogi2008finding,chazal2006sampling,chazal2008smooth,attali2011vietoris}.
In the seminal work of Niyogi et al. \cite{niyogi2008finding} the authors showed 
for a sufficiently small radius, controlled by the reach $r(X)$, the Euclidean thickening of a dense enough sample deformation retracts onto the submanifold $X\subset \R^N$. 
Beyond smooth submanifolds, a broader class of shapes includes the compact Euclidean subsets  with a positive reach \cite{federer}, more generally positive $\mu$-reach $r_\mu(X)$ \cite{chazal2006sampling} for an appropriately chosen value of the parameter $\mu$ (see Section \ref{S:mu-reach}).
The known reconstruction results for such shapes were obtained in \cite{attali2011vietoris} and \cite{kim2020homotopy} using both Vietoris--Rips and \v{C}ech complexes.
Another very important class of shapes is embedded metric graphs see the work in \cite{majhi2023vietoris,lecci2014statistical} for reconstruction via the Vietoris--Rips complexes. 

\subsection{Acknowledgements} The first and third author acknowledge the partial support by Louisiana Board of Regents
Targeted Enhancement Grant 090ENH-21. No data or code were generated or used in this research.

\section{Preliminaries}\label{sec:prelim}
In this section, we provide some preliminary notation and definitions used throughout the paper. 

\subsection{Length Spaces}
Let $(X,d)$ be a metric space. For a (continuous) path $\gamma:[a,b]\to X$ in $X$, we first define the notion of its length, denoted $L_d(\gamma)$, induced by the metric $d$.

We call $T=\{t_i\}_{i=0}^k$ a partition of $[a,b]$ if:
\[
a=t_0 < t_1 < t_2 < \ldots < t_k = b.
\]
Then, the length of $\gamma$ can be defined as:
\[
L_d(\gamma)\coloneqq \sup \sum_{i=0}^{k-1}d(\gamma(t_i), \gamma(t_{i+1})),
\]
where the supremum is taken over all partitions $T$ of $[a,b]$. A path $\gamma$ is called \emph{rectifiable} if
its length is finite. When the choice of the metric $d$ is obvious from the context, we drop the subscript $d$ and denote the induced length simply by $L(\cdot)$.
We note from \cite[Proposition 2.3.4]{burago2022course} that $L$ is lower semi-continuous, a property that finds its use later in the paper.
\begin{lemma}[Lower Semi-Continuity]\label{lem:burago-liminf}
Let $(X,d)$ be a metric space.
If a sequence of rectifiable paths $\{\gamma_i\colon[a,b]\to X\}_{i=1}^\infty$ and a path $\gamma:[a,b]\to X$ are such that $\lim\limits_{i\to\infty}d(\gamma_i(t),\gamma(t))=0$ for every $t\in [a,b]$, then $\liminf L(\gamma_i) \geq L(\gamma)$.    
\end{lemma}
The notion of the length of paths in $X$, in turn, associates a distance between any two points of $X$.
\begin{definition}[Induced Length Metric]\label{def:induced-length}
For any two points $x_1,x_2\in X$, the \emph{induced length metric}, denoted $d^L(x_1,x_2)$, defined by the infimum of lengths of paths joining $x_1,x_2$:
\[
d^L(x_1,x_2)\coloneqq\inf\left\{L(\gamma)\mid\gamma:[a,b]\to X,\gamma(a)=x_1, \gamma(b)=x_2\right\}.
\]
\end{definition}
Although not always finite, $d^L$ defines a metric on $X$. 
A metric space $(X,d)$ is called a \emph{length space} and its metric $d$ is length metric if $d$ coincides with the induced length metric $d^L$. For any metric space $(X,d)$, its induced length metric space $(X,d^L)$ is always a length space. 
Moreover, $(X,d)$ is called a \emph{complete length space} if for any two points $x_1,x_2\in X$, there exists a shortest path $\gamma:[a,b]\to X$ joining them, i.e.,
the induced distance $d^L(x_1,x_2)=L(\gamma)$. 

A subset $A$ of a length space $(X,d)$ is called \emph{convex} if for any two points $a_1,a_2\in A$, there exists a unique shortest path joining $a_1$ and $a_2$ lies entirely in $A$. 

\subsection{Spaces of curvature bounded above }\label{S:cat(k)}
Alexandrov originally studied these spaces in the early 1900s to generalize the notion of Riemannian curvature to length spaces \cite{aleksandrov1951theorem,alexandrov1948intrinsic}. 
For $\kappa\in\R$, a space $X$ is a complete length space such that within a convexity radius the geodesic triangles in $X$ are \emph{not fatter} than corresponding \emph{comparison triangles} in the model surface of constant sectional curvature $\kappa$.
Such space is called $\cat(\kappa)$ if the triangle comparison condition holds globally. 
Consequently, the class includes closed Riemannian manifolds, locally-finite metric graphs and various stratified spaces.
The following formal definition is in the style of \cite{burago2022course,lang1997jung}. 
\begin{definition}[The space of curvature $\leq \kappa$]\label{def:space-curvature-bounded}
For $\kappa\in\R$, a complete length space $(X,d)$ has the curvature $\leq \kappa$ if for any $x_1,x_2,x_3\in X$ (in case $\kappa>0$ within a convex ball (cf. Definition \ref{def:conv-rad}) additionally $d(x_1,x_2)+d(x_2,x_3)+d(x_3,x_1)<2\pi/\sqrt{\kappa}$ ) the following \emph{triangle comparison} property holds:
{\em
$d(y,x_2)\leq d(\overline{y},\overline{x}_2)$ for any point $y$ on a shortest path joining $x_1$ and $x_3$, where $\overline{x}_i$'s are points on the surface of constant curvature $\kappa$ with $d(\overline{x}_i,\overline{x}_j)=d(x_i, x_j)$ and $\overline{y}$ is the point on the shortest path joining $\overline{x}_1$ and $\overline{x}_3$ with $d(\overline{x}_1,\overline{y})=d(x_1, y)$.
}
\end{definition}

\subsection{Hausdorff and Gromov--Hausdorff Distances}
Let $A$ and $B$ be compact, non-empty subsets of a metric space $(X, d)$. Roughly speaking, the Hausdorff distance between $A$ and $B$ is the infimum radius $\eta>0$ that suffices for the thickenings $A^\eta$ and $B^\eta$ to cover $B$ and $A$, respectively. Here, the thickening are taken in $X$ under the metric $d$, e.g,
\[
A^{\eta}\coloneqq \bigcup_{a\in A} \{x \in X \mid d(x, a) < \eta\}.
\]
More concretely, we have the following definition.
\begin{definition}[Hausdorff Distance]\label{def:dH}
The \emph{Hausdorff distance} between the $A$ and $B$, denoted $d^X_H(A, B)$, is defined as
\[
d^X_H(A, B) \coloneqq \inf \{\eta > 0 \mid B \subset A^{\eta} \text{ and } A \subset B^{\eta}\}.  
\]
In case $X\subset\mathbb{R}^N$ and $A,B,X$ are all equipped with the Euclidean metric, we simply write $d_H(A,B)$.
\end{definition}

A \emph{correspondence} $\C$ between any two metric spaces $(X,d_X)$ and $(Y,d_Y)$ is defined as a subset of $X\times Y$ such that
\begin{enumerate}[(a)]
\item for any $x\in X$, there exists $y\in Y$ such that $(x,y)\in\C$, and
\item for any $y\in Y$, there exists $x\in X$ such that $(x,y)\in\C$.
\end{enumerate}
We denote the set of all correspondences between $X,Y$ by $\C(X,Y)$. The
\emph{distortion} of a correspondence $\C\in\C(X,Y)$ is defined as:
\[
\mathrm{dist}(\C)\coloneqq\sup_{(x_1,y_1),(x_2,y_2)\in\C}
|d_X(x_1,x_2)-d_Y(y_1,y_2)|.
\]
\begin{definition}[Gromov--Hausdorff Distance]\label{def:gh} Let $(X,d_X)$ and $(Y,d_Y)$ be two compact metric spaces. The \emph{Gromov--Hausdorff} distance between $X$ and $Y$, denoted by $d_{GH}(X,Y)$, is defined as:
\[
d_{GH}(X,Y)\coloneqq
\frac{1}{2}\left[\inf_{\C\in\C(X,Y)}\mathrm{dist}(\C)\right].
\]
\end{definition}

A slightly weaker and more local notion is $(\eps,R)$-Gromov--Hausdorff closeness for constants $\eps,R>0$. 
The notion has already been used in \cite{aanjaneya2012metric} in the context of reconstruction of metric graphs and is referred to as {\em $(\eps,R)$-approximation} in their work.
However, we use a slightly different version to match the traditional definition of the Gromov--Hausdorff distance.
\begin{definition}[$(\eps,R)$-Gromov--Hausdorff closeness]\label{def:eps-R-approx}
For $\eps,R$ positive, two metric spaces $(X,d_X)$ and $(Y,d_Y)$ are said to be $(\eps,R)$-Gromov--Hausdorff close if there exists a correspondence $\C\subset X\times Y$ such that for any $(x_1,y_1),(x_2,y_2)\in\C$
with $\min\{d_X(x_1,x_2),d_Y(y_1,y_2)\}\leq R$, we have
\[
|d_X(x_1,x_2)-d_Y(y_1,y_2)|\leq 2\eps.
\]
\end{definition}
To see that this is a weaker notion, we note that $d_{GH}(X, Y)<\eps$ implies that the two metric spaces are $(\eps,R)$-Gromov--Hausdorff close for any $R>0$. However, an $(\eps,R)$-Gromov--Hausdorff closeness may not always imply $\eps$-closeness in the Gromov--Hausdorff distance; see the following example.
\begin{example}
Let $W\subset\R^2$ be the wedge $\{(x,y)\in\mathbb{R}^2\colon y=|x|\text{ and } x\leq1\}$, and let $X$ and $Y$ both be equal to $W$ as sets, but endowed with the Euclidean metric $\|\cdot\|$ and the induced length metric $d^L$, respectively.
For any $\eps>0$, the diagonal correspondence $\C=\{(w,w)\mid w\in W\}\subset X\times Y$ is an $(\tfrac{1}{2}\varepsilon,2\varepsilon)$–correspondence, 
 but the diameters of $X$ and $Y$ are respectively $2$ and $2\sqrt{2}$, resulting in $d_{GH}(X,Y)\geq\frac{1}{2}(2\sqrt{2}-2)>0$.
\end{example}

\subsection{Abstract Simplicial Complexes and Geometric Realizations}
\begin{definition}[Abstract Simplicial Complex]\label{def:simp-comp}
An \emph{abstract simplicial complex} is a set $\K$ closed under subset inclusion. 
That is, for each non-empty, finite $\sigma \in \K$, if $\sigma' \subset \sigma$, then $\sigma' \in \K$. 
\end{definition}
Each $\sigma \in \K$ is called a \emph{simplex} of $\K$; a subset $\sigma' \subset \sigma$ is called a \emph{face} of the $\sigma$. 
If the cardinality of a simplex $\sigma\in\K$ is $k+1$, then we call it a $k$-simplex and define its \emph{dimension}, denoted $\dim(\sigma)$, to be $k$. 
The collection of all $0$-dimensional simplices is called the vertex set of $\K$.

Note $\K$ may have sufficient combinatorial structure for our purposes, but not enough topological structure. 
We may add such topological structure by defining a \textit{geometric realization} of $\K$, denoted as $|\K|$ to be the topological space $\K$ under some choice of a topology on $\K$; find more details in \cite{spanier1995algebraic}. For notational convenience, we denote by $\K$ both an abstract simplicial complex and its geometric realization.

A simplicial complex $\K$ is called a \emph{pure $m$--complex} if every simplex
of $\K$ is a face of an $m$--simplex. A simplicial complex $\K$ is called a
\emph{flag complex} if $\sigma$ is a simplex of $\K$ whenever every pair of
vertices in $\sigma$ is a simplex of $\K$.

\begin{definition}[Vietoris--Rips Complex]\label{def:rips}
Let $(X, d)$ be a metric space. Then, the \textit{Vietoris--Rips complex} of $X$ at scale $\beta>0$, denoted $\Ri_\beta(X)$, is an abstract simplicial complex such that the $n$-simplex $[x_0,x_1,\ldots,x_n]$ is in $\Ri_\beta(X)$ if and only if $d(x_i,x_j)<\beta$ for all $0\leq i,j\leq n$. 
\end{definition}


\subsection{Barycentric Subdivision}
The \emph{barycenter}, denoted $\widehat{\sigma}_m$, of an $m$-simplex
$\sigma_m=[v_0$, $v_1$, $\ldots$, $v_m]$ of $\K$ is the point of $\langle\sigma_m\rangle$ such
that $\widehat{\sigma}_m(v_i)=\frac{1}{m+1}$ for all $0\leq i\leq m$. Using
linearity of simplices, a more convenient way of writing this is: 
\begin{equation}\label{eq:barycenter-point}
\widehat{\sigma}_m=\sum_{i=0}^m \frac{1}{m+1}v_i.    
\end{equation}
For a simplicial complex $\K$, its \emph{barycentric subdivision}, denoted $\sd{\K}$, is a simplicial complex such that 
\begin{enumerate}[(i)]
\item the vertices of $\sd{\K}$ are the barycenters of simplices of $\K$
\item the simplices of $\sd{\K}$ are (non-empty) finite sets $[\widehat\sigma_0,\widehat\sigma_1,\ldots,\widehat\sigma_m]$ such that $\sigma_{i-1}\prec\sigma_i$ for $1\leq i\leq m$ and $\sigma_i\in\K$
\item the linear map $\sd\colon\sd{\K}\to\K$ sending each vertex of $\sd{\K}$ to the corresponding point of $\K$ is a homeomorphism.
\end{enumerate}
The following result from \cite{majhi2023vietoris} proves a crucial ingredient for our proofs. 
\begin{lemma}[Commuting Diagram \cite{majhi2023vietoris}]\label{lem:barycenter-map-lemma} Let $\K$ be a pure $m$--complex and $\L$ a flag complex. Let
$f\colon\K\to\L$ and $g\colon\sd\K\to\L$ be simplicial maps such that 
\begin{enumerate}
\item $g(v)=f(v)$ for every vertex $v$ of $\K$,
\item $f(\sigma)\cup g(\widehat{\sigma})$ is a simplex of $\L$ whenever
$\sigma$ is a simplex of $\K$.
\end{enumerate}
Then, the following diagram commutes up to homotopy: 
\begin{equation*}
\begin{tikzpicture} [baseline=(current  bounding  box.center)]
\node (k1) at (-2,-2) {$\sd{\K}$};
\node (k2) at (2,-2) {$\K$};
\node (k3) at (0,0) {$\L$};
\draw[map] (k2) to node[auto,swap] {${\sd^{-1}}$} (k1);
\draw[map,swap] (k1) to node[auto,swap] {$g$} (k3);
\draw[map,swap] (k2) to node[auto] {$f$} (k3);
\end{tikzpicture}
\end{equation*}
where $\sd$ is the linear homeomorphism sending each vertex of $\sd{\K}$ to the corresponding point of $\K$.
\end{lemma}

\begin{theorem}[Whitehead Theorem, \cite{hatcher2002book}]\label{thm:whitehead}
Let $X$ and $Y$ be topological spaces, homotopy equivalent to CW--complexes\footnote{which are allowed to be uncountably infinite, as we encounter e.g. in Theorem \ref{thm:hausmann} and in Theorem \ref{thm:main-rips}.}. 
If $f: X \to Y$ is a map inducing isomorphisms on all homotopy groups, i.e., $\pi_m(f_*): \pi_m(X) \to \pi_m(Y)$ is an isomorphism for all $m \geq 0$, then $f$ is a homotopy equivalence. 
\end{theorem}

\section{Latschev's Theorem For Spaces with Curvature $\leq \kappa$}
This section extends Latschev's theorem to the spaces of  Definition~\ref{def:space-curvature-bounded}. 
In addition, we discuss how notions like convexity radius and results like Jung's and Hausmann's theorem make sense for such spaces.

\subsection{Hausmann's Theorem}
We begin with the definition of the convexity radius of a length space.
\begin{definition}[Convexity Radius]\label{def:conv-rad} The \emph{convexity radius} of a length space $(X,d)$, denoted $\rho(X)$, is defined as the infimum of the set of radii of the largest convex metric balls across the points of $X$. Formally,
\[
\rho(X)=\inf_{x\in X}\sup\;\{r\geq0\mid\mathbb B_X(x,s)
\text{ is convex for all }0<s<r\}.	
\]
\end{definition}
It is well known (\cite[Lemma 3.4]{Ballmann1995LecturesOS}) that a metric ball in the space of curvature $\leq \kappa$ is always convex for a sufficiently small radius. As a consequence, $\rho(X)>0$ for compact locally $\cat(\kappa)$ spaces.

Although Hausmann defines of convexity radius in the context of Riemannian manifolds, his definition and the results of \cite{hausmann1995vietoris} are equally valid for length spaces. 
Hausmann's definition is slightly different from ours, however, the alteration does not tamper with the veracity of Hausmann's theorem:
\begin{theorem}[Hausmann's Theorem~\cite{hausmann1995vietoris}]
\label{thm:hausmann}
Let $X$ be a compact space of curvature $\leq \kappa$ with convexity radius $\rho(X)$. 
For any $0<\beta<\rho(X)$, we have a homotopy equivalence $\Ri_\beta(X) \simeq X$.

Furthermore, if $0<\beta\le \beta'<\rho(X)$, then the inclusion $\Ri_\beta(X)\hookrightarrow\Ri_{\beta'}(X)$ is a homotopy equivalence~\cite{virk2021rips}.
\end{theorem}

In this context, we note the works of \cite{Adamaszek2018,Adams_2019}, \cite{Virk2021-lc}, and \cite{Lemez2022-au} for variations and extensions of Hausmann's theorem and alternative proofs.

Throughout the section, we assume that $(X,d)$ is a compact space of Definition \ref{def:space-curvature-bounded}; consequently, $\rho(X)>0$.
We define the quantity
\begin{equation}\label{eq:Delta(X)}
\Delta(X)\coloneqq\begin{cases}
\min\left\{ \frac{\pi}{4\sqrt{\kappa}}, \rho(X)\right\}, &\text{ if } \kappa >0, \\
\rho(X), &\text{ if } \kappa\leq 0.
\end{cases}
\end{equation}

\subsection{Jung's Theorem}\label{sec:jung}
We first state the classical Jung's Theorem in the Euclidean space, and then in in the setting of curvature $\leq \kappa$.
For a bounded subset $A$ of any metric space $(X,d)$, we define its \emph{circumradius} to be the radius of the smallest (closed) metric ball enclosing $A$. More formally,
\[
    \rad(A)\coloneqq\inf_{x\in X}\ \sup_{a\in A}d(a, x).
\]
A point $x\in X$ satisfying $\max_{a\in A}d(a ,x)=\rad(A)$ is called a \emph{circumcenter} of $A$, and is denoted by $c(A)$. For $A$ compact, a circumcenter may not uniquely exist. 
We warn the reader that these terms may not always match their usual meaning in plane geometry.

For a subset $A\subset\R^N$ with exactly $n+1$ distinct points, the classical Jung's theorem
\cite[Theorem 2.6]{danzer1963helly} states that $c(A)$ is unique and  $\rad(A)\leq\sqrt{\frac{n}{2(n+1)}}\diam(A)$.

The above bound on the circumradius has been extended in both directions; see for example \cite{Dekster1985AnEO,Dekster1995TheJT,dekster1997jung,lang1997jung}. 
Particularly relevant is the extension of Jung's theorem to $\cat(\kappa)$ spaces \cite{lang1997jung}.

\begin{theorem}[Generalized Jung's Theorem \cite{lang1997jung}]
\label{thm:jungs_catk}
Let $X$ be a complete $\cat(\kappa)$ space with a non-empty subset $A$ of $X$ with at most $n+1$ distinct points for any positive integer $n$. In the case where $\kappa > 0$, assume that the circumradius of $A$ satisfies $\operatorname{rad}(A) < \frac{\pi}{2\sqrt{\kappa}}$. Then, there is a unique circumcenter $c(A)$ in $X$.
Moreover,  
\begin{align}\label{eq:jungs}
    \diam(A) \geq \begin{cases}
    \frac{2}{\sqrt{\kappa}}\sin^{-1}\left(\sqrt{\frac{n+1}{2n}}\sin{(\sqrt{\kappa}\rad(A))}\right) & \text{ if } \kappa > 0, \\
    2\sqrt{\dfrac{n+1}{2n}}\rad(A) & \text{ if } \kappa = 0, \\
    \frac{2}{\sqrt{-\kappa}}\sinh^{-1}\left(\sqrt{\frac{n+1}{2n}}\sinh{(\sqrt{-
    \kappa}\rad(A))}\right) & \text{ if } \kappa < 0.
    \end{cases}
\end{align} 
\end{theorem}
The following lemma provides a convenient bound for the circumradius. 
A similar bound in the context of Riemannian manifolds can be found in \cite{majhi2024demystifying}.
\begin{lemma}[Diameter Estimate]\label{lem:circum-center}
Let $X$ be a compact  space with curvature $\leq \kappa$. 
Let $A$ be a non-empty, finite subset with $\diam(A)<\Delta(X)$.
Then,
\begin{enumerate}[(i)]
    \item $A$ has a unique circumcenter $c(A)$ in $X$,
    \item the diameter $\diam\bigl(A\bigr)\geq \dfrac{4}{3}\rad(A)$.
    \item if $B\subset A$ is a non-empty subset, then $d(c(A),c(B))\leq\dfrac{3}{4}\diam(A)$.
\end{enumerate}
\end{lemma}
\noindent See Appendix for the proof.
\subsection{Proofs of Theorems \ref{thm:latschev-cat} and \ref{thm:main-stability}}\label{S:proof-ThmB}
\mainLatschev*
\begin{proof}
Since $S$ is $(\xi\beta,\beta)$-Gromov--Hausdorff close to $X$, there exists a correspondence $\C\subset X\times S$ such that for any $(x_1,s_1),(x_2,s_2)\in\C$, we have
\begin{equation}\label{eq:distortion}    
\min\{d_X(x_1,x_2),d_S(s_1,s_2)\}\leq \beta
\implies |d_X(x_1,x_2)-d_S(s_1,s_2)|\leq 2\xi\beta.
\end{equation}
The correspondence induces (possibly non-continuous and non-unique) maps $\phi:X\to S$ and $\psi:S\to X$ such that  
$(x,\phi(x))\in\C$ for all $x\in X$ and $(\psi(s),s)\in\C$ for all $s\in S$, respectively.
We now argue that the maps on the vertex sets extend to the simplicial maps of Vietoris--Rips complexes, as follows
\begin{equation}\label{eq:chain}
\Ri_{(1-2\xi)\beta}(X)\xrightarrow{\quad\phi\quad}
\Ri_\beta(S)\xrightarrow{\quad\psi\quad}\Ri_{(1+2\xi)\beta}(X).
\end{equation}
{\bf Simplicial Maps.\ }
To see that $\phi$ is a simplicial map, take an $l$--simplex
$\sigma_l=[x_0,x_1,\ldots,x_l]$ in $\Ri_{(1-2\xi)\beta}(X)$. By definition, we must have
$d_X(x_i,x_j)<(1-2\xi)\beta$ for any $0\leq i,j\leq l$. 
From~\eqref{eq:distortion}, we then obtain
\[
d_S(\phi(x_i),\phi(x_j))\leq d_X(x_i,x_j)+2\xi\beta
<(1-2\xi)\beta+2\xi\beta=\beta.
\]
Thus the image $\phi(\sigma_l)=[\phi(x_0),\phi(x_1),\ldots,\phi(x_l)]$ is a
simplex of $\Ri_\beta(S)$.
Similarly, $\psi$ can be shown to be a simplicial map.

The rest of the proof establishes that $\phi$ induces isomorphisms on all homotopy groups. 
{\bf Injectivity.\ }
We argue that the composition $(\psi\circ\phi)$ is contiguous to the natural inclusion:
$\Ri_{(1-2\xi)\beta}(X)\xhookrightarrow{\quad
\iota\quad}\Ri_{(1+2\xi)\beta}(X)$. 
To prove the claim, take an $l$--simplex
$\sigma_l=[x_0,x_1,\ldots,x_l]$ in $\Ri_{(1-2\xi)\beta}(X)$. 
Thus,
$d_X(x_i,x_j)<(1-2\xi)\beta$ for all $0\leq i,j\leq l$. 
We then note from \eqref{eq:distortion} that
\begin{align*}
d_X((\psi\circ\phi)(x_i),x_j) &=d_X(\psi(\phi(x_i)),x_j) \\
&\leq d_S(\phi(x_i),\phi(x_j))+2\xi\beta \\
&\leq d_X(x_i,x_j)+2\xi\beta+2\xi\beta \\
&<(1-2\xi)\beta+4\xi\beta=(1+2\xi)\beta.
\end{align*}
This implies that $(\psi\circ\phi)(\sigma_l)\cup\iota(\sigma_l)$ is a simplex of
$\Ri_{(1+2\xi)\beta}(X)$. 
Since $\sigma_l$ is an arbitrary simplex, the
simplicial maps $(\psi\circ\phi)$ and $\iota$ are contiguous. 
		
Since $(1+2\xi)\beta<\Delta(X)\leq\rho(X)$, Hausmann's Theorem~\ref{thm:hausmann} implies that
there $\iota$ is a homotopy equivalence.
Hence, the induced homomorphism $\iota_*$ on the homotopy groups are isomorphisms. On the
other hand, we already have $\iota\simeq\psi\circ\phi$. Therefore,
the induced homomorphisms $\left(\psi_*\circ\phi_*\right)$ are also isomorphisms, implying that $\phi_*$ is an injective homomorphism on $\pi_m\left(\Ri_{(1-2\xi)\beta}(X)\right)$ for all $m\geq0$. 

\noindent {\bf Surjectivity.\ }
Next, we argue that $\phi_*$ is also a surjective homomorphism on $\pi_m\left(\Ri_{(1-2\xi)\beta}(X)\right)$ for any $m\geq0$.

As observed in Proposition~\ref{prop:path-connected}, both $\Ri_{(1-2\xi)\beta}(X)$ and $\Ri_\beta(S)$ are path-connected. Thus the result holds for $m=0$. 	

For $m\geq1$, let us take an abstract simplicial complex $\K$ such that
$\K$ is a triangulation of the $m$-dimensional sphere $\S^m$. Note that
$\K$ can be assumed to be a pure $m$--complex. In order to show surjectivity of $\phi_*$ on
$\pi_m\left(\Ri_{(1-2\xi)\beta}(X)\right)$, we start with a simplicial
map $g:\K\to\Ri_{\beta}(S)$ representing any given class in $\pi_m\left(\Ri_{(1-2\xi)\beta}(X)\right)$; such $g$ exists thanks to the Simplicial Approximation Theorem \cite[Theorem
16.5]{munkres2018elements}.
Next, we argue that there must exist a simplicial map\footnote{Here $\sd{\K}=\sd^1{\K}$ is a single subdivision of $\K$.}
$\widetilde{g}:\sd{\K}\to\Ri_{(1-2\xi)\beta}(X)$ such that the following
diagram commutes up to homotopy:
\begin{equation}\label{eqn:diag}
\begin{tikzpicture} [baseline=(current  bounding  box.center)]
	\node (k1) at (-2,0) {$\Ri_{(1-2\xi)\beta}(X)$};
	\node (k2) at (2,0) {$\Ri_\beta(S)$};
	\node (k3) at (2,-2) {$\K$};
	\node (k4) at (-2,-2) {$\sd{\K}$};
	\draw[map] (k1) to node[auto] {$\phi$} (k2);
	\draw[map,swap] (k3) to node[auto,swap] {$g$} (k2);
	\draw[map,dashed] (k4) to node[auto,swap] {$\widetilde{g}$} (k1);
	\draw[map, <-, dashed] (k4) to node[auto] {$\sd^{-1}$} (k3);
\end{tikzpicture}
\end{equation}
where the linear homeomorphism $\sd:\sd{\K}\to\K$ maps each vertex of $\sd{\K}$ to the corresponding point of $\K$. 
We note that each vertex of $\sd{\K}$ is the barycenter, $\widehat{\sigma}$, of
a simplex $\sigma$ of $\K$. 
To construct the simplicial map $\widetilde{g}:\sd{\K}\to\Ri_{(1-2\xi)\beta}(X)$, we define it on the vertices of $\sd{\K}$ first, and prove that the vertex map extends to a simplicial map. 

Let $\sigma_l=[v_0,v_1,\ldots,v_l]$ be an $l$--simplex of $\K$. Since $g$ is a
simplicial map, the image $g(\sigma_l)=[g(v_0),g(v_1),\ldots,g(v_l)]$ is a
simplex of $\Ri_\beta(S)$, hence a subset of $S$ with
$\diam_S(g(\sigma_l))<\beta$. For each $0\leq j\leq l$, there exists $x_j\in X$
such that $(x_j,g(v_j))\in\C$. 
We denote by $\sigma_j':=[x_0,x_1,\ldots,x_j]$ for $0\leq j\leq l$.
We note from~\eqref{eq:distortion}, for later, that the diameter of $\sigma_j'$ is less than $\Delta(X)$: 
\begin{equation}\label{eq:sigma}
\diam_X\bigl(\sigma_j'\bigr)\leq\diam_S\bigl(g(\sigma_j)\bigr)+2\xi\beta
<\beta+2\xi\beta=(1+2\xi)\beta<\Delta(X).
\end{equation}
We then define the vertex map
\[
\widetilde{g}(\widehat{\sigma}_l)\coloneqq c(\sigma_l'),
\]
where $c(\sigma_l')\in X$ is a circumcenter of $\sigma_l'$. 
Due to the diameter bound in~\eqref{eq:sigma}, Lemma~\ref{lem:circum-center} implies that a circumcenter of
$\sigma_l'$ exists. 

To see that $\widetilde{g}$ extends to a simplicial map, we consider a typical
$l$--simplex $\tau_l=[\widehat{\sigma}_0,\ldots,\widehat{\sigma}_l]$, of
$\sd{\K}$, where $\sigma_{i-1}\prec\sigma_i$ for $1\leq i\leq l$ and
$\sigma_i\in\K$. Now,
\[
\begin{split}
\diam_X(\widetilde{g}(\tau_l))
&=\diam_X([c(\sigma_0'),c(\sigma_1'),\ldots,c(\sigma_l')) \\
&=\max_{0\leq i<j\leq l}\{d_X(c(\sigma_i'),c(\sigma_j'))\}.
\end{split}
\]
Thus,
\[
\begin{split}
\diam_X(\widetilde{g}(\tau_l)) &\leq\max_{0\leq j\leq l}\left\{ \left(\frac{3}{4}\right)\diam_X(\sigma_j')
\right\}, \\
&\quad\quad\quad\text{ by Lemma~\ref{lem:circum-center} as }\diam{\sigma_j'}<\Delta(X) 
\end{split}
\]
\[
\begin{split}
&=\frac{3}{4}\diam_X(\sigma_l') <\frac{3}{4}(1+2\xi)\beta,\text{ from~\eqref{eq:sigma}} \\
	&=(1-2\xi)\beta-(1-14\xi)\beta/4 \\
	&\leq(1-2\xi)\beta,\text{ since }\xi\leq\tfrac{1}{14}.
\end{split}
\]
Thus $\widetilde{g}(\tau_l)$ is a simplex of $\Ri_{(1-2\xi)\beta}(X)$. This implies that $\widetilde{g}$ is a simplicial map. 
	
We lastly invoke Lemma~\ref{lem:barycenter-map-lemma} to show that the diagram commutes up to homotopy. 
We need to argue that the simplicial maps $g$ and
$(\phi\circ\widetilde{g})$ satisfy the conditions of Lemma~\ref{lem:barycenter-map-lemma}:
\begin{enumerate}[(a)]
\item For any vertex $v\in\K$, 
\[
(\phi\circ\widetilde{g})(v)=g(v).
\]

\item For any simplex $\sigma_m=[v_0,v_1,\ldots,v_m]$ of $\K$, we have for $0\leq
j\leq m$:
\begin{align*}
d_S(g(v_j),(\phi\circ\widetilde{g})(\widehat{\sigma}_m))	
&=d_S(g(v_j),\phi(c(\sigma_m'))) \\
&\leq d_X\left(x_j,c(\sigma_m')\right)+2\xi\beta,
\text{ since }(x_j,g(v_j))\in\C \\
&\leq\frac{3}{4}\diam_X(\sigma_m')+2\xi\beta,
\text{ by Lemma~\ref{lem:circum-center} as }x_j=c({x_j}) \\
&<\frac{3}{4}(1+2\xi)\beta+2\xi\beta,
\text{ from~\eqref{eq:sigma}} \\
&=\beta-(1-14\xi)\beta/4 \leq\beta,\text{ since }\xi\leq\tfrac{1}{14}.
\end{align*}
\end{enumerate}
Thus, $g(\sigma_m)\cup(\phi\circ\widetilde{g})(\widehat{\sigma}_m)$ is a simplex
of $\Ri_\beta(S)$. 
Therefore, Lemma~\ref{lem:barycenter-map-lemma} implies that the diagram commutes. 
Since $\K=\S^m$ and $g$ is arbitrary, we conclude that $\phi$ induces a surjective homomorphism.

\noindent {\bf Homotopy Equivalence\ } For any $m\geq0$, therefore,
\[
\phi_*:\pi_m\left(\Ri_{(1-2\xi)\beta}(X)\right)
\to\pi_m\left(\Ri_\beta(S)\right).	
\]
is an isomorphism. 
It follows from Whitehead's theorem that $\phi$ is a homotopy equivalence.
Since $\Ri_{(1-2\xi)\beta}(X)$ is homotopy equivalent to $X$, we
conclude that $\Ri_\beta(S)\simeq X$.
\end{proof}
\begin{remark}
Complexes $\Ri_{(1-2\xi)\beta}(X)\simeq X$ and $\Ri_\beta(S)$ are a priori infinite complexes, and we know from \cite{Milnor1957-uu} that Whitehead's theorem applies to such complexes. 
On the other hand, $X$ is compact of positive convexity radius $\rho(X)>0$, therefore, the cover by open balls $\{B_{d_X}(x,\rho)\}$ is a good cover with a finite subcover. 
In turn, the Nerve Lemma \cite{hatcher2002book} implies $X$ has a homotopy type of a finite complex, thus Theorems \ref{thm:main-rips}, \ref{thm:latschev-cat}, \ref{thm:main-stability}, \ref{thm:main-approx} are concerned with finite complexes up to homotopy.
\end{remark}

\noindent We immediately obtain our main stability result.
\mainStability*
\begin{proof}
Treating $X$ or $X'$ as the sample metric space $(S,d_S)$ in Theorem \ref{thm:latschev-cat}, we have for instance $\Ri_{\beta}(X')\simeq X$. Since $\beta<\rho(X)$, Haussman's Theorem \ref{thm:hausmann} implies  $X'\simeq\Ri_{\beta}(X')$.
\end{proof}

\section{Euclidean Path Metrics and Geodesic Subspaces of \texorpdfstring{$\R^N$}{RN}}\label{sec:path-metric} 
This section introduces the notion of the Euclidean path metrics and presents some of their properties relevant to our Euclidean shape reconstruction scheme.
In our setting, the underlying hidden shape $X\subset\R^N$ is equipped with the Euclidean metric. 
We denote by $d^L_X$, the induced length metric on $X$; see Definition~\ref{def:induced-length}. 
We consider, for reconstruction, shapes $X\subset\R^N$ so that under the induced length metric $(X,d^L_X)$ it is a complete length space---also called a \emph{geodesic subspace} of $\R^N$.

We use the notation $\Ri^L_\beta(X)$ to denote the Vietoris--Rips $X$ under the length metric $d^L$.

\begin{remark}
For the Euclidean shape $X$, the metric $(X,d^L_X)$ produces a \emph{finer} topology on $X$ than $(X,\|\cdot\|)$. 
The two topologies are not generally the same even when $X$ is compact in the Euclidean topology. One way to force the equivalence is to impose a finite global distortion (Definition~\ref{def:dist}); see also \cite[Remark 2.3]{fasy2022reconstruction}. 
However, the equivalence of the two topologies is not required to build the theory of this paper.  
\end{remark}

Unless otherwise indicated, we always assume that a geodesic subspace $X\subset\R^N$ is {\bf\emph{compact}} and \emph{\bf\emph{path-connected}} in the $d^L_X$ topology. 
As a consequence, so are they in the Euclidean topology.

\begin{definition}[Global Distortion]\label{def:dist}
The (global) \emph{distortion of embedding} or simply \emph{distortion} (following \cite{gromov1999metric}) is defined as follows
\begin{equation*}\label{eq:dist}
\delta(X) \coloneqq \displaystyle \sup_{a \neq b \in X}\dfrac{d^L_X(a,b)}{||a-b||}.
\end{equation*}    
\end{definition}

\subsection{The Path Metric and Its Properties}
For any $Y\subset\mathbb{R}^N$, the notion of an $\eps$--path was introduced in \cite{majhi2023vietoris}, \cite{fasy2022reconstruction} as a family of piecewise paths in $\mathbb{R}^N$. The family gives rise to an alternative metric on $Y$---different from the Euclidean submetric.

A (non-empty) subset $Y\subset\R^N$ comes equipped with the standard Euclidean metric, given by the Euclidean norm $\norm{\ \cdot\ }$. 
We define another metric, denoted $d^\eps_Y$,  using the pairwise Euclidean distances of points in $Y$. For a positive number $\eps$, we first introduce the notion of an $\eps$--path.
\begin{definition}[$\eps$--Path]\label{def:eps-path}
Let $Y\subset\R^d$ be non-empty and $\eps>0$ a number. 
For $a,b\in Y$, an \emph{$\eps$--path} of $Y$ from $a$ to $b$ is a finite sequence $P=\{y_i\}_{i=0}^{k+1}\subseteq Y$ such that $y_0=a$, $y_{k+1}=b$, and $\norm{y_i-y_{i+1}}<\eps$ for all $i=0,1,\ldots,k$. 
\end{definition}
We now define the length of the path by
\begin{equation*}\label{eq:length_of_path}
L(P)\coloneqq\sum_{i=0}^k\norm{y_i-y_{i+1}}.    
\end{equation*}
We denote the set of all $\eps$--paths of $Y$ by $\mathscr{P}_Y^\eps$.
We finally introduce the $\eps$--path metric on $Y$.
\begin{definition}[$d^\eps_Y$--Metric]\label{def:d-eps} Let $Y\subset\R^d$ be non-empty and $\eps>0$ a number. The $\eps$--path metric on $Y$, denoted $d^\eps_Y$, between any $a,b\in Y$ is defined by
\begin{equation*}\label{eq:d^eps}
d^\eps_Y(a,b)\coloneqq\inf\left\{L(P)\mid P\in \mathscr{P}^\eps_Y\text{ is an $\eps$--path of $Y$ between $a$ and $b$}\right\}.
\end{equation*}
\end{definition}
In general, for $d^\eps_Y(a,b)$ to be a well defined metric on $Y$ we require the Euclidean thickening $Y^{\frac{\eps}{2}}$ to be path-connected as the following proposition shows.
\begin{proposition}[Path-metric Space]
Let $Y\subset\R^N$ and $\eps>0$ be a number such that the Euclidean thickening $Y^{\frac{\eps}{2}}$ is path-connected. Then, $(Y,d^\eps_Y)$ is a metric space.    
\end{proposition}
\begin{proof}
Let $a,b\in Y$. We show that there must exist an $\eps$--path of $Y$ joining them. 
Since $Y^{\frac{\eps}{2}}=\cup_{y\in Y}\mathbb B(y,\frac{\eps}{2})$ is path-connected, there exists a sequence of points $\{y_i\}_{i=0}^{k+1}\subset Y$ such that $y_0=a$, $y_{k+1}=b$, and $\mathbb B(y_{i}, \frac{\eps}{2})\cap  \mathbb B(y_{i+1}, \frac{\eps}{2})$ is non-empty for each $i\in\{1,\ldots,k\}$.
As a consequence, $\norm{y_{i}-y_{i+1}}<\eps$ for each $i\in\{1,\ldots,k\}$. Thus, $\{y_i\}_{i=0}^k$ is indeed an $\eps$--path between $a$ and $b$. 
\end{proof}

In the metric space $(Y, d^\eps_Y)$, we denote the diameter of a subset $A\subset Y$ by $\diam_\eps(A)$. 
For any scale $\beta>0$, the Vietoris--Rips complex of $(Y,d^\eps_Y)$ is denoted by $\Ri^\eps_\beta(Y)$. 

We now present some interesting properties of path-metrics.
\begin{proposition}[Non-decreasing]\label{prop:d_X^eps-nonincreasing}
Let $Y\subset\R^N$ be a subset and $a,b\in Y$. 
Then, $d_Y^\eps(a,b)$ is a non-increasing function of $\eps$.  
\end{proposition}
\begin{proof}
Let $\eps_1<\eps_2$ and $0<\eta<\eps_2-\eps_1$ be arbitrary.
By the definition of $\eps$--paths, there exists an $\eps_1$--path $P$ joining $a$ and $b$ such that  $L(P)\leq d_Y^{\eps_1}(a,b)+\eta$.
Since $\eps_1<\eps_2$, note that $P$ is also an $\eps_2$--path joining $a,b$. Thus, 
\[
d_Y^{\eps_2}(a,b)\leq L(P)\leq d_Y^{\eps_1}(a,b)+\eta.
\]
Since $\eta$ is arbitrary, we conclude that $d_Y^{\eps_2}(a,b)\leq d_Y^{\eps_1}(a,b)$.
\end{proof}
The following Proposition from \cite[Prop. 4.6]{majhi2023vietoris} is of central importance. 
\begin{proposition}[Comparison of Path Metrics]\label{prop:d^esp-d^L-estimate} Let $X \subset \mathbb{R}^N$ be a geodesic subspace and $S \subset \mathbb{R}^N$ such that 
$d_H(S, X) < \frac{1}{2}\xi\varepsilon$ for some $\xi \in (0,1)$ and $\varepsilon > 0$. 
For any $x_1, x_2 \in X$ and corresponding $s_1, s_2 \in S$ with $||x_1-s_1||, ||x_2-s_2||<\frac{1}{2}\xi\varepsilon$, we have
\begin{equation}\label{eq:d^eps-d^L-si-eps}
\norm{s_1-s_2}\leq d^\eps_S(s_1,s_2)\leq \dfrac{d^L_X(x_1,x_2) + \xi\varepsilon}{1-\xi}.
\end{equation}
\end{proposition}
Note the first inequality comes from the triangle inequality, whereas the second inequality directly follows from \cite[Prop. 4.6]{majhi2023vietoris}. 

\subsection{Stability of Path Metrics}
The following result demonstrates the stability of the path metric under Hausdorff perturbation. Although we do not directly apply it in later arguments, we record it here as of independent interest.
\begin{theorem}[Stability of Path Metrics]\label{thm:path-metric-stab}
Let $X,S\subset\R^N$ be any subsets with $d_H(X,S)<\frac{1}{2}\xi\eps$ for some $\xi\in(0,1)$ and $\eps>0$. Let $p,q\in S$ and $p',q'\in X$ such that $\norm{p-p'}<\frac{1}{2}\xi\eps$ and $\norm{q-q'}<\frac{1}{2}\xi\eps$. Then, $d^{(2+\xi)\eps}_X(p',q')\leq(1+\xi)d^\eps_S(p,q)$.
\end{theorem}
\begin{proof}
The result trivially follows if $d^\eps_S(p,q)<\eps$. Thus, we assume that $d^\eps_S(p,q)\geq\eps$.

Let $P=\{x_i\}_{i=0}^{l+1}$ be an $\eps$--path in $S$ joining $p,q$, i.e. $x_0=p$ and $x_{l+1}=q$. 
Since $d^\eps_S(p,q)\geq\eps$, there must exist a subsequence $\{x_{i_j}\}_{j=0}^{k+1}$ of $P$ such that $i_0=0$, $i_{k+1}=l+1$, and 
\begin{equation}\label{eq:eps-x_i}
\eps\leq\sum_{i=i_j}^{i_{j+1}}\norm{x_i-x_{i+1}}<2\eps
\text{ for any }0\leq j\leq k.
\end{equation}
As a consequence, the total length
\[
L(P)=\sum_{i=0}^{l+1}\norm{x_i-x_{i+1}}=\sum_{j=0}^{k}\sum_{i=i_j}^{i_{j+1}}\norm{x_i-x_{i+1}}\geq\sum_{j=0}^{k}\eps= (k+1)\eps.
\]
Now, for each $j$, there is a corresponding point $x_{i_j}'\in X$ such that $\norm{x_{i_j}-x'_{i_j}}<\frac{1}{2}\xi\eps$. 
We choose $x'_{i_0}=p'$ and $x'_{i_{k+1}}=q'$.
For any $0\leq j\leq k$, we observe first from the triangle inequality and afterward from \eqref{eq:eps-x_i} that
\[
\norm{x'_{i_j}-x'_{i_{j+1}}}
\leq\norm{x_{i_j}-x_{i_{j+1}}}+\xi\eps
\leq\sum_{i=i_j}^{i_{j+1}}\norm{x_i-x_{i+1}}+\xi\eps
<2\eps+\xi\eps=(2+\xi)\eps.
\]
Thus, the sequence of points $P'=\{x'_{i_j}\}_{j=0}^{k+1}$ forms a $(2+\xi)\eps$--path in $X$ joining $p'$ and $q'$. 
Moreover, its length
\begin{align*}
L(P')&=\sum_{j=0}^k\norm{x'_{i_j}-x'_{i_{j+1}}}
\leq\sum_{j=0}^k\left(\norm{x_{i_j}-x_{i_{j+1}}}+\xi\eps\right) 
\leq\sum_{j=0}^k\left(\sum_{i=i_j}^{i_{j+1}}\norm{x_i-x_{i+1}}+\xi\eps\right) \\
&=L(P) + \sum_{j=0}^k\xi\eps =L(P) + (k+1)\xi\eps 
\leq L(P) + \xi L(P)=(1+\xi)L(P).
\end{align*}
By the definition of path metric, we can write
$d_X^{(2+\xi)\eps}(p', q')\leq L(P')\leq(1+\xi)L(P)$. This is true for any arbitrary $\eps$--path $P$ joining $p$ and $q$. Therefore, we conclude that 
\[d^{(2+\xi)\eps}_X(p', q')\leq(1+\xi)d^\eps_S(p,q).\qedhere\]
\end{proof}

\subsection{Convergence of Path Metrics}\label{sec:path-metric-conv}
Intuitively, as $\eps$ go to $0$, the length of the $\eps$--path between two points $x_1,x_2\in X$ converges to converges to the intrinsic distance $d^L_X(x_1,x_2)$. We make the intuition more precise in the following result, which is reminiscent of Gromov--Hausdorff convergence Theorem in \cite[p. 265, Theorem 7.5.1]{burago2022course} in the specific setting of thickenings $X^\eps$ and Hausdorff closeness.

\begin{theorem}[Convergence of Path Metric]\label{thm:geodesic-convergence} 
Let $X\subset\R^N$ be a geodesic subspace. 
Then, $d^\eps_X$ converges to $d^L_X$ uniformly on $X\times X$ as $\eps\to 0$. 
\end{theorem}
Recall $X^\eps$ for any $\eps>0$, is the $\eps$--thickening of $X$ in $\mathbb{R}^N$.
We denote by $d^L_{X^\eps}$ the induced length metric on $X^\eps$. Moreover, for any $x_1,x_2\in X$, we must have that $d^L_{X^\eps}(x_1,x_2) \leq d^L_X(x_1,x_2)$.

Theorem~\ref{thm:geodesic-convergence} follows almost immediately from the following technical lemma concerning the convergence of induced length metrics.
\begin{lemma}[Convergence of Induced Length Metric]\label{lem:eps-geodesics}
Let $X\subset\R^N$ be a geodesic subspace and $x_1, x_2\in X$.
Then, $d^{L}_{X^\eps}(x_1,x_2)\leq d^{L}_X(x_1,x_2)$, and
\begin{equation*}
d^{L}_{X^\eps}(x_1,x_2)\longrightarrow d^{L}_X(x_1,x_2)\text{ as }\eps\to 0.
\end{equation*}
\end{lemma}

\begin{proof}[Proof of Lemma \ref{lem:eps-geodesics}]
For each small $\eps>0$, let $\gamma_{\eps}$ be a path in $X^{\eps}$ connecting $x_1$ and $x_2$ such that $L(\gamma_\eps)\leq d^L_{X^\eps}(x_1, x_2)+\eps$. Consequently, the lengths of the paths are bounded above; let $L$ be an upper bound, i.e., $L\geq L(\gamma_\eps)$ for all $\eps$.

Since $X$ is a compact length space, the sequence $\{\gamma_\eps\}$ is uniformly bounded.
Using the diagonal argument (as used in the proof of Arzela-Ascolli's them \cite{rudinAnalysis}), we obtain a subsequence $\{\gamma_{\eps_i}\}$ pointwise convergent to a function $\gamma: I \longrightarrow X$. We now argue that the convergence is uniform and that $L(\gamma_{\eps_i})\longrightarrow L(\gamma)$. 

We (re)-parameterize each $\gamma_{\eps_i}(t)$ with the arc-length parameter $t_i=\phi_i(s)$ as follows:
\[
\begin{split}
\gamma_i&\coloneqq\gamma_{\eps_i}:[0,L]\longrightarrow X^{2\eps},\ \gamma_i(0)=x_1,\\
\gamma_i(s) & = \gamma_{\eps_i}(\phi_i(s)),\quad s\in [0,L(\gamma_i)],\\
\gamma_i(s) & = x_2,\quad s\in [L(\gamma_i),L].
\end{split}
\]
Let $\overline{\eps}>0$ be any small number. 
Then, we can choose $n\in\mathbb N$ such that $\frac{L}{n}<\frac 13\overline{\eps}$.
We then subdivide $[0,L]$ in equal intervals:
\begin{equation*}\label{eq:s_i-s}
0=s_0<s_1<\ldots< s_n =L\text{ and }s_j=s_{j-1}+\frac{L}{n}=j\frac{L}{n}\text{ for }1\leq j\leq n.
\end{equation*}
Thanks to the pointwise convergence for any given $\overline{\eps}>0$, there exists an  $N(\overline{\eps})\in\mathbb N$, such that for every $k,m>N(\overline{\eps})$:
\begin{equation}\label{eq:gamma_km-cauchy}
\|\gamma_k(s_j)-\gamma_m(s_j)\|< \tfrac 13\overline{\eps},\qquad \text{for every}\quad j=0,\ldots, n.
\end{equation}
(i.e. Cauchy's condition holds at all points in $\{s_j\}_{j=0}^n$). 
We note now that the condition extends to any $s\in [s_j,s_{j+1}]$:
\begin{align*}
\|\gamma_k(s)-\gamma_m(s)\| 
&\leq \|\gamma_k(s)-\gamma_k(s_j)\|+
\|\gamma_k(s_j)-\gamma_m(s_j)\|+
\|\gamma_m(s_j)-\gamma_m(s)\|\\
&< L\left(\gamma_k\big|_{[s_j,s_{j+1}]}\right) +\tfrac{\overline{\eps}}{3} 
+L\left(\gamma_m\big|_{[s_j,s_{j+1}]}\right),\text{ since Cauchy at }s_j\\
&\leq \|s_j-s_{j+1}\|+\tfrac{\overline{\eps}}{3} 
+\|s_j-s_{j+1}\|,\text{ due to arc-length parametrization}\\
&<\tfrac{\overline{\eps}}{3}+\tfrac{\overline{\eps}}{3}+\tfrac{\overline{\eps}}{3}=\overline{\eps}.
\end{align*}
This proves that the sequence $\{\gamma_k\}$ is uniformly Cauchy, and thus uniformly convergent to $\gamma$. Consequently, $\gamma$ is a continuous path connecting $x_1$ and $x_2$.


Applying Lemma~\ref{lem:burago-liminf} on $(X^\eps, \|\cdot\|)$, we get from the semi-continuity of the induced length
\[
\lim\inf L(\gamma_{\eps_i}) \geq L(\gamma).
\]
As a result, the length $L(\gamma)\leq\liminf L(\gamma_{\eps_i})\leq \liminf [d^L_{X^\eps}(x_1, x_2)+\eps]$. 
On the other hand, $d^L_{X^\eps}(x_1, x_2)\leq d^L_{X}(x_1, x_2)\leq L(\gamma)$. Thus, we have
\[
\limsup\limits_{\eps\to0} d^L_{X^\eps}(x_1, x_2) \leq L(\gamma) \leq \liminf\limits_{\eps\to0} [d^L_{X^\eps}(x_1, x_2)+\eps].
\]
Therefore, the limit exists: $L(\gamma)=\lim d^L_{X^\eps}(x_1,x_2)$.
\end{proof}

Now, we prove the convergence result.
\begin{proof}[Proof of Theorem \ref{thm:geodesic-convergence}]
For any $x_1, x_2\in X$, the pointwise convergence  follows from Lemma \ref{lem:eps-geodesics} by sending $\eps\to 0$ in
\[
d^{L}_{X^{2\eps}}(x_1,x_2)\leq d^\varepsilon_X(x_1,x_2)\leq d^L_X(x_1,x_2).
\]
Here, the first inequality comes from the fact that an $\eps$-path on $X$ is a continuous path joining $x_1$ and $x_2$ in $X^{2\eps}$, and the second inequality follows from the fact that we may always choose an $\eps$-path along the shortest path connecting $x_1$ and $x_2$. 

We note that $d^\eps_X$ is a monotone non-decreasing function of $\eps$ (as $\eps$ decreases). Moreover, both $d^L_X$ and $d^\eps_X$ are continuous w.r.t. the $d^L$ topology on $X\times X$.
Now, the uniform convergence follows from Dini's theorem (see \cite{Bartle2018-lk} for example) and the assumption that $X\times X$ is compact w.r.t. the $d^L_X$ topology.
\end{proof}
As an immediate consequence, we note the following Gromov--Hausdorff convergence.
\begin{proposition}
For a geodesic subspace $X\subset\R^N$, the path metric $(X,d^\eps_X)$ converges to the induced length metric $(X,d^L_X)$ in the Gromov--Hausdorff distance as $\eps\to 0$. 
\end{proposition}
\begin{proof}
Fix $\overline{\eps}>0$.
From the uniform convergence in Theorem~\ref{thm:geodesic-convergence}, we can choose an $\eps>0$ such that
\[
\sup_{x_1,x_2\in X} |d^\eps_X(x_1,x_2)-d^L_X(x_1,x_2)|\leq\overline{\eps}
\]
Now, choosing the diagonal
correspondence between $X$ and itself, note from its (Gromov--Hausdorff) distortion that
\[
d_{GH}((X,d^\eps_X),(X,d^L_X))\leq\frac{1}{2}\sup_{x_1,x_2\in X} |d^\eps_X(x_1,x_2)-d^L_X(x_1,x_2)|\leq\overline{\eps}.
\]
This concludes the result.
\end{proof}

\section{Reconstruction of Euclidean Subspaces under Hausdorff Distance}\label{sec:recon}
We define our most important sampling parameter: large scale distortion. 


Let $X\subset\R^N$ be a geodesic subspace. 
The large scale distortion of $X$, denoted $\delta^\eps_R(X)$, is parametrized by $\eps>0$ and $R>0$.
It is defined as the supremum of the ratio of the induced length of metrics on $X$ and (Euclidean) $\eps$-thickening of $X$, between points of $X$ that are at least $R$ distance away.
More formally, we present the following definition.
\begin{definition}[Large Scale Distortion]\label{def:rest-dist}
For $\eps>0$ and $R>0$, the \emph{large scale distortion} or \emph{$(\eps,R)$-distortion} of a geodesic subspace $X\subset\R^N$ is defined as
\begin{equation}\label{eq:delta-eps-R}
\delta^\eps_R(X)\coloneqq \sup_{d^L_X(x_1,x_2)\geq R} \frac{d_X^L(x_1,x_2)}{d^L_{X^{\eps}}(x_1,x_2)}
.
\end{equation}
\end{definition}
It immediately follows from the definition that $\delta^\eps_R(X)$ is a non-decreasing function of $\eps$ and that $\delta^\eps_R(X)$ is a non-increasing function of $R$. 
Moreover, $\delta^\eps_R(X)\leq\delta(X)$, the global distortion of embedding of $X$. Indeed, we have
\[
\lim_{\substack{\eps\to\infty\\R\to0}}\delta^\eps_R(X)=\delta(X).
\]

\begin{remark}
A slightly modified version of Definition~\ref{def:rest-dist} given as 
\[
\widehat{\delta}^\eps_R(X)\coloneqq \sup_{d^L_X(x_1,x_2)\geq R} \frac{d_X^L(x_1,x_2)}{d^\eps_{X}(x_1,x_2)}
\]
can be also applied in Theorem~\ref{thm:main-rips} and \ref{thm:main-approx}.
It appears that this version of large scale distortion is more ameanable to computations, which the authors plan to further investiagate in future works. 
\end{remark}

Now, we present the most important convergence property of large scale distortion.
\begin{proposition}[Convergence of Restricted Distoriton]\label{prop:distortion_to_1}
Let $X$ be a geodesic subspace of $\R^N$. 
For any $R>0$, we have $\delta^\eps_R(X)\to1$ as $\eps\to0$.
\end{proposition}
\begin{proof}
Let $x_1,x_2\in X$ be such that $x_1\neq x_2$. We have already established in Theorem~\ref{thm:geodesic-convergence} that the ratio 
\[
\delta^\eps(x_1,x_2)\coloneqq\frac{d^L(x_1,x_2)}{d^L_{X^\eps}(x_1, x_2)}\to1\text{ as }\eps\to0.
\]
We argue that $\delta^\eps(x_1,x_2)$ is also continuous at $x_1, x_2\in X$ w.r.t. $(X,d^L)$.
The numerator, length metric $d^L_X:X\times X\to\R$, is naturally continuous with respect to the $(X,d^L)$.
Similarly, the denominator $d^L_{X^\eps}\colon X\times X\to \R$ is continuous w.r.t $(X,d^L_{X^\eps})$, the subspace topology of the induced length metric on $X^\eps$. Since $d^L_{X^\eps}(x_1, x_2)\leq d^L_X(x_1, x_2)$, the latter topology $X$ in finer. Consequently, the denominator $d^L_{X^\eps}\colon X\times X\to \R$ remains continuous when assumed the latter topology $(X,d^L_X)$ on the domain. Next, note that  $d^L_{X^\eps}$ is bounded away from zero when restricted to 
\[
    X^2_R=X\times X - \Delta^\eps_R = \left\{(x_1,x_2)\in X\times X \mid\ d^L_X(x_1,x_2)\geq R\right\},
\]
indeed, from compactness of $X^2_R$, $\min_{X^2_R} d^L_{X^\eps}$ is achieved at some $(a_0,b_0)\in X^2_R$, if 
$d^L_{X^\eps}(a_0,b_0)=0$, then $a_0=b_0$ which yields a contradiction.
As a result, the ratio $\delta^\eps$ is continuous on $X^2_R$.
Since we have pointwise convergence of $\{\delta^\eps:X^2_R\to\R\}$ to $1$, and it is a non-descreasing function of $\eps$, Dini's theorem (see \cite{Bartle2018-lk} for example) implies uniform convergence. \qedhere

\end{proof}

We prove the following technical result.
\begin{proposition}[Large Scale Distortion and Path Metrics]\label{prop:delta_path}
Let $X$ be a geodesic subspace of $\R^N$ and $S\subset\R^N$ compact with $d_H(X, S)<\alpha$. Let $\eps>2\alpha$ and $R$ be positive numbers. For any $x_1,x_2\in X$ such that $d^L_X(x_1,x_2)\geq R$ and $s_1,s_2\in S$ such that $\|x_i-s_i\|<\alpha$ for $i=1,2$, we must have
\[
d^L_X(x_1, x_2)\leq\delta^\eps_R(X)(d^\eps_S(s_1, s_2)+2\alpha).
\]
\end{proposition}
\begin{proof}
We consider the path $\gamma:[a,b]\to\R^N$ joining $x_1$ and $x_2$ by concatenating three pieces: the segment $[x_1,s_1]$ (of length $\leq \alpha$) with an $\eps$--path connecting $s_1$ and $s_2$ (of length $d^\eps_S(s_1, s_2)$), and the segment $[s_2,x_2)]$ (of length $\leq \alpha$).

Since $\alpha<\eps/2$, it's not difficult to see that the image of $\gamma$ lies entirely in $X^\eps$. 
This implies that $d^L_{X^\eps}(x_1, x_2)\leq L(\gamma)\leq(d^\eps_S(s_1,s_2)+2\alpha)$.

From the definition of large scale distortion, we now get
\[
d^L_X(x_1, x_2)\leq\delta^\eps_R(X)d^L_{X^\eps}(x_1, x_2)\leq\delta^\eps_R(X)(d^\eps_S(s_1,s_2)+2\alpha).\qedhere
\]
\end{proof}

\subsection{Proofs of Theorems \ref{thm:main-approx} and \ref{thm:main-rips}}

\mainApprox*
\begin{proof}
Let $S\subset\R^N$ be compact with $d_H(X,S)<\frac{1}{2}\xi\eps$.
We define the following relation $\C\subset X\times S$:
\[
\C=\left\{(x,s)\in X\times S\colon \|x-s\|<\tfrac{1}{2}\xi\eps\right\}.
\]
Since $d_H(X,S)<\frac{1}{2}\xi\eps$, the relation $\C$ is, in fact, a correspondence.

Take arbitrary $(x_1,s_1),(x_2,s_2)\in\C$ such that $\min\{d_X(x_1,x_2),d_S(s_1,s_2)\}\leq\beta$.
Without any loss of generality, we further assume that $d_S(s_1,s_2)\leq\beta$.
We now argue that:
\[
\left\lvert d^L_X(x_1,x_2)-d^\eps_S(s_1,s_2)\right\rvert\leq2\xi\beta.
\]
We consider the following two cases:
\begin{enumerate}[(1)]
\item If $d^L_X(x_1,x_2)\leq d^\eps_S(s_1,s_2)$, then from the right inequality of Proposition~\ref{prop:d^esp-d^L-estimate} we get
\begin{align*}
\left\lvert d^L_X(x_1,x_2)-d^\eps_S(s_1,s_2)\right\rvert
&=d^\eps_S(s_1,s_2)-d^L_X(x_1,x_2)
\leq d^\eps_S(s_1,s_2)-[(1-\xi)d^\eps_S(s_1,s_2)-\xi\eps]\\
&=\xi d^\eps_S(s_1,s_2)+\xi\eps\leq2\xi\beta.
\end{align*}
The last inequality is due to the fact that $\eps\leq\beta$ and $d^\eps_S(s_1,s_2)\leq\beta$.

\item If $d^L_X(x_1,x_2)\geq d^\eps_S(s_1,s_2)$, then
$\left\lvert d^L_X(x_1,x_2)-d^\eps_S(s_1,s_2)\right\rvert
=d^L_X(x_1,x_2)-d^\eps_S(s_1,s_2)$. 
We consider the following two sub-cases:

If $d^L_X(x_1,x_2)<2\xi\beta$, then 
\[
\left\lvert d^L_X(x_1,x_2)-d^\eps_S(s_1,s_2)\right\rvert
=d^L_X(x_1,x_2)-d^\eps_S(s_1,s_2)
\leq d^L_X(x_1,x_2)
<2\xi\beta.\]

If $d^L_X(x_1,x_2)\geq2\xi\beta$, then Proposition~\ref{prop:delta_path} implies that 
\[
d^L_X(x_1,x_2)\leq\delta^\eps_{2\xi\beta}(X)(d^\eps_S(s_1, s_2)+\xi\eps).
\]
As a result,
\begin{align*}
\left\lvert d^L_X(x_1,x_2)-d^\eps_X(s_1,s_2)\right\rvert
&=d^L_X(x_1,x_2)-d^\eps_S(s_1,s_2)\\
&\leq\delta^\eps_{2\xi\beta}(X)(d^\eps_S(s_1, s_2)+\xi\eps)-d^\eps_S(s_1,s_2)\\
&\leq\tfrac{1+2\xi}{1+\xi}(d^\eps_S(s_1, s_2)+\xi\eps)-d^\eps_S(s_1,s_2)\\
&=\tfrac{\xi}{1+\xi}d^\eps_S(s_1, s_2)+\tfrac{\xi(1+2\xi)}{1+\xi}\eps\\
&\leq\tfrac{\xi}{1+\xi}\beta+\tfrac{\xi(1+2\xi)}{1+\xi}\beta\text{ since }d^\eps_S(s_1, s_2)\leq\beta\text{ and }\eps\leq\beta\\
&=2\xi\beta.
\end{align*}
\end{enumerate}
Combining all the cases, we conclude that
\begin{equation*}
\left\lvert d^L_X(x_1,x_2)-d^\eps_S(s_1,s_2)\right\rvert\leq2\xi\beta.
\end{equation*}
Hence, $(S,d^\eps_S)$ is an $(\xi\beta,\beta)$-Gromov--Hausdorff close to $(X,d^L)$.
\end{proof}
\noindent Combining Theorems \ref{thm:latschev-cat} and \ref{thm:main-approx} yields our main result:

\mainRips*

\section{More on Distortion and \texorpdfstring{$\mu$}{mu}--reach}\label{S:mu-reach}
In this section, we compare the well-known sampling parameter $r_{\mu}(X)$, the $\mu$--reach of a compact subset $X \subset \R^N$ (see Chazal, Cohen-Steiner, and Lieutier \cite{chazal2006sampling}) and the large scale distortion parameter $\delta^{\eps}_R(X)$ used in our main results above. We intend to show an example that favors our large scale distortion parameter over the $\mu$--reach, in case $X$ is a subspace of curvature $\leq \kappa$ in the intrinsic length metric of $\R^N$.

Let us denote by $R_X:\R^N\rightarrow [0,\infty)$ the Euclidean distance function defined as $R_X(a)=\min_{x\in X}\|a-x\|$.
Recall from \cite{chazal2006sampling} that the generalized gradient function of $R_X$ is defined by 
\begin{align}\label{eq:generalized_gradient}
\nabla_X(x) = \dfrac{x - \Theta_X(x)}{R_X(x)},
\end{align}
where $\Theta_X(x)$ is the point closest to $x$ in the convex hull of the set: 
\begin{align}\label{eq:gamma_k}
\Gamma_X(x) = \{y \in X \mid ||x-y||=R_X(x) \}.
\end{align}
It follows that $x_0\in\R^N$ is a critical point of $R_X$ (i.e. $\nabla_X(x_0)=0$) if and only if $x_0$ is in $\Gamma_K(x_0)$;
in the case $X$ is a submanifold, $R_X$ is differentiable and $\nabla_X$ is the same as the classical gradient. Next, the \textit{critical function} $\chi_X: (0, \infty) \rightarrow \mathbb{R}^+$ of $X$ is defined as  
\begin{align}\label{eq:critical_function}
\chi_X(d) = \displaystyle \inf_{y\in R^{-1}_X(d)}||\nabla_X (y) ||
\end{align}
i.e. for each fixed value of the distance to $X$, $\chi_X(d)$ is the infimum of $\|\nabla_X(\,\cdot\,)\|$ along the $d$--level set of $R_X$.
Finally, the $\mu$-reach of a compact subset $X \subset \mathbb{R}^N$ is defined to be 
\begin{align}\label{eq:mu_reach}
r_{\mu}(X) = \displaystyle \inf \{d \mid \chi_{X}(d) < \mu\}.
\end{align}
\begin{figure}[ht]
    \centering
    \includegraphics[width=0.3\linewidth]{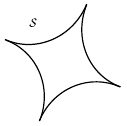}\qquad \includegraphics[width=0.4\linewidth]{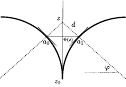}
    \caption{(left) Star--shaped subspace $\mathcal{S}$ ({\em Ninja's Star}) (right) $r_{\mu}(\mathcal{S})=0$ at the cusp.}
    \label{fig:n-star-cusp}
\end{figure}
Consider the cusped star--shaped subspace $\mathcal{S}$ of $\R^2$ as shown in Figure \ref{fig:n-star-cusp} (left) which is a union of four-quarter circular arcs. Let $z$ be a point near the cusp, which sits on the bisector line of the cusp, within distance $d$ to $\mathcal{S}$. From  \eqref{eq:generalized_gradient} and Figure~\ref{fig:n-star-cusp} (right), we obtain $\Gamma_{\mathcal{S}}(z)=\{a_0,a_1\}$ and $\Theta(z)$ is a midpoint of the segment $[a_0,a_1]$, thus
\[
\nabla_{\mathcal{S}}(z)=\frac{z-\Theta(z)}{d},\quad \text{and}\quad \|\nabla_{\mathcal{S}}(z)\|=|\sin(\varphi)|.
\]
Clearly, as $z$ approaches the cusp point $z_0$, we have $d\longrightarrow 0$ and $\varphi\longrightarrow 0$, thus $\|\nabla_{\mathcal{S}}(z)\|\longrightarrow 0$. Applying definitions \eqref{eq:critical_function} and \eqref{eq:mu_reach} we conclude that for any value of $\mu$, the $\mu$--reach $r_{\mu}(\mathcal{S})$ vanishes. On the other hand, we clearly have $0<\delta^\eps_R(\mathcal{S})<\infty$ for any value of $R>0$ and $\eps>0$, also the intrinsic convexity radius of $\mathcal{S}$ is positive and the curvature is negative. Therefore, Theorem~\ref{thm:main-rips} is applicable in the case of $X=\mathcal{S}$ or other similar ``cuspy'' shapes, see Figure \ref{fig:ufo} for another example in dimension $2$.
In the context of the previous works \cite{fasy2022reconstruction}, note that $\delta(X)$ is infinite at the cups of $\mathcal{S}$ (or other types of cusps as well), see  
 \cite{tran2024thesis} for alternative approaches to this issue and the concept of $\alpha$--distortion.

\begin{figure}
	\centering
	\includegraphics[width=0.5\linewidth]{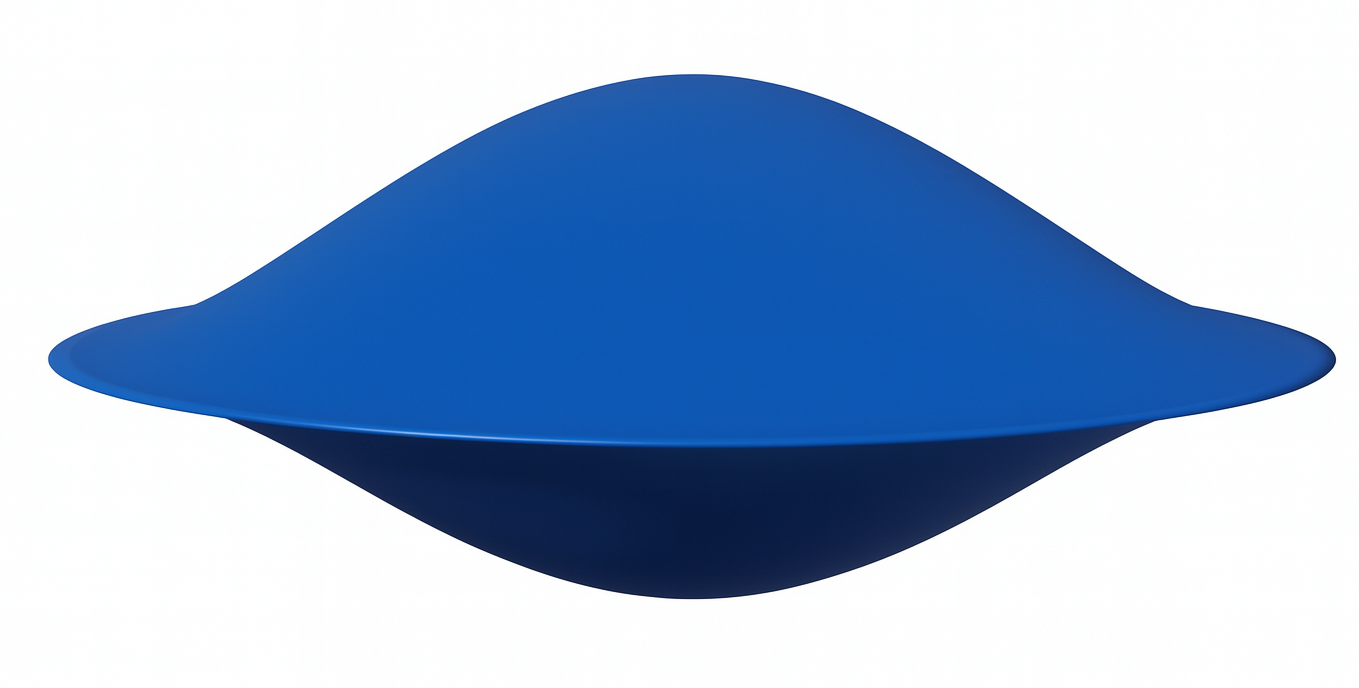}
	\caption{A $2$--sphere ({\em UFO surface}) with vanishing $\mu$--reach due to the ``cuspy'' rim.}
	\label{fig:ufo}
\end{figure}

\appendix
\section{Additional Proofs} \label{apx:proofs}
\begin{proof}[Proof of Lemma~\ref{lem:circum-center}] 
First note that the closed ball $\overline{\mathbb B(a,\diam(A))}$ for any $a\in A$ contains $A$ entirely, implying that $\rad(A)\leq\diam(A)$. 
Since
$\diam(A)<\Delta(X)\leq\frac{\pi}{4\sqrt{\kappa}}$, we therefore have
$\rad(A)\in\left[0,\frac{\pi}{4\sqrt{\kappa}}\right]$. 
Thus, Part (i) follows immediately from Jung's Theorem~\ref{thm:jungs_catk}.

Let $A$ contain $n+1$ points for $n\geq2$\footnote{When $n=1$, the result is trivial since the circumcenter of two points is the midpoint and the circumradius is half the diameter.}. 
We divide the proof of Part (ii) into the following
two cases depending on the sign of $\kappa$.

\paragraph{Case 1 ($\kappa\leq0$)}
From \eqref{eq:jungs}, we get 
\[
\diam(A)\geq2\rad(A)\sqrt{\frac{n+1}{2n}}\geq2\rad(A)\frac{1}{\sqrt{2}}\geq\frac{4}{3}\rad(A).
\]

\paragraph{Case 2 ($\kappa>0$)}
Let us now define the function
$J(r)\coloneqq\frac{2}{\sqrt{\kappa}}\sin^{-1}\left(\sqrt{\frac{n+1}{2n}}\sin{\sqrt{\kappa}r}\right)$,
and show that $J(r)/r$ is a strictly decreasing function for
$r\in\left(0,\frac{\pi}{4\sqrt{\kappa}}\right)$. We take its derivative:
\begin{equation}\label{eq:J}
\frac{d}{dr}(J(r)/r)=\frac{rJ'(r)-J(r)}{r^2}=\frac{f(r)}{r^2},
\end{equation}
where $f(r)\coloneqq rJ'(r)-J(r)$. We observe that $f(r)$ is strictly decreasing,
since
\[
f'(r)=rJ''(r)=r\frac{2\sqrt{\kappa}\sqrt{\frac{n+1}{2n}}\sin{\sqrt{\kappa}r}}{\left(1-\frac{n+1}{2n}\sin^2{\sqrt{\kappa}r}\right)^{3/2}}\left(\frac{1-n}{2n}\right)<0
\]
for $r\in\left(0,\frac{\pi}{4\sqrt{\kappa}}\right)$ and $n\geq2$. Then it
follows from $f(0)=0$ that $f(r)<0$ on
$r\in\left(0,\frac{\pi}{4\sqrt{\kappa}}\right)$. 
	
Thus, \eqref{eq:J} implies that the function $J(r)/r$ is a strictly decreasing
function for $r\in\left(0,\frac{\pi}{4\sqrt{\kappa}}\right)$. Its minimum, as a
result, is attained at $r=\frac{\pi}{4\sqrt{\kappa}}$. Thus, for
$r\in\left(0,\frac{\pi}{4\sqrt{\kappa}}\right]$ we get
\[
\frac{J(r)}{r}\geq\frac{2}{\sqrt{\kappa}}\sin^{-1}\left(\sqrt{\frac{n+1}{2n}}\sin{\frac{\pi}{4}}\right)\frac{4\sqrt{\kappa}}{\pi}
\geq\frac{8}{\pi}\sin^{-1}\left(\sqrt{\frac{1}{2}}\sin{\frac{\pi}{4}}\right)
=\frac{8}{\pi}\frac{\pi}{6}=\frac{4}{3}.
\] 
On the other hand, \eqref{eq:jungs} implies that $\diam(A)\geq J(\rad(A))$
for $\rad(A)\in\left[0,\frac{\pi}{4\sqrt{\kappa}}\right]$. Hence,
$\diam{A}\geq\frac{4}{3}\rad(A)$ as desired.

For Part (iii), we note from Theorem~\ref{thm:jungs_catk} that both the circumcenters $c(A)$ and $c(B)$ exist. 
The definition of $\rad(A)$ then implies that
$A\subseteq\overline{\mathbb B(c(A),\rad{A})}$, the closed metric ball of radius $\rad(A)$ centered at $c(A)$. 
From the assumptions, we also have 
\[
\rad(A)\leq\diam(A)<\Delta(X)\leq\rho(X).
\]
Therefore, $X'\coloneqq\overline{\mathbb B(c(A),\rad{A})}$ is convex.
Applying Theorem~\ref{thm:jungs_catk} for $B$ as a subset of $X'$, we conclude that $c(B)\in X'$.
Therefore,
\[
d(c(B),c(A))<\rad(A)\leq\frac{3}{4}\diam(A).\qedhere
\]
\end{proof}

\begin{proposition}[Path-connectedness]\label{prop:path-connected} 
Let $(X,d_X)$ be path-connected and $(S,d_S)$ be $(\eps,R)$--Gromov--Hausdorff close to $X$ for some $\eps,R>0$. 
Then the geometric complex of $\Ri_{R+2\eps}(S)$ is
path-connected.
\end{proposition}

\begin{proof}
Let $\C$ be an $(\eps,R)$-correspondence between $S$ and $X$.
Let $s,s'\in S$ be arbitrary points.
Then there exist points $x,x'\in X$ such that $(s,x),(s',x')\in\C$. Since $X$ is assumed to be path-connected, so is $\Ri_R(X)$. As a result, there exists a sequence
$\{x_i\}_{i=0}^{m+1}\subset X$ forming a path in $\Ri_R(X)$ joining $x$ and $x'$. In other words, $x_0=x$, $x_{m+1}=x'$, and $d_X(x_i,x_{i+1})<R$ for $0\leq i\leq m$. There is also a corresponding
sequence $\{s_i\}_{i=0}^{m+1}\subset S$ such that $s_0=s$, $s_{m+1}=s'$, and
$(s_i,x_i)\in\C$ for all $i$. We note that
\[
d_S(s_i,s_{i+1})\leq d_X(x_i,x_{i+1})+2\eps<R+2\eps.
\]
Thus, the sequence $\{s_i\}$ produces a path in $\Ri_{R+2\eps}(S)$ joining $s$ and $s'$. We conclude that the geometric complex of $\Ri_{R+2\eps}(S)$ is path-connected.
\end{proof}

\end{document}